\documentclass[11pt]{article}
\usepackage{amsmath, amssymb}
\usepackage{mathrsfs}
\usepackage[margin=1in]{geometry}
\usepackage{amsmath,amssymb,amsthm,mathtools}
\usepackage{enumitem}
\usepackage[hidelinks]{hyperref}
\usepackage{indentfirst}
\usepackage[compress,numbers,square]{natbib}

\theoremstyle{definition}
\newtheorem{theorem}{Theorem}[section]
\newtheorem{lemma}[theorem]{Lemma}
\newtheorem{corollary}[theorem]{Corollary}
\newtheorem{proposition}[theorem]{Proposition}

\newtheorem{conjecture}{Conjecture}
\newtheorem{definition}[theorem]{Definition}

\DeclareMathOperator{\diam}{diam}

\newcommand{\E}{\mathcal E}
\newcommand{\Kemeny}{\mathcal K}

\begin{document}

\begin{center}
			{{\huge On two conjectures concerning Kemeny's constant of graphs}} \\[18pt]
			{\Large  Wei Li$^{1}$, Wensheng Sun$^{2}$, Yujun Yang$^{1}$*\footnotetext{*Corresponding author at E-mail address: yangyj@ytu.edu.cn}}\\[6pt]
			{ \footnotesize  1. School of Mathematics and Information Sciences,Yantai University,Yantai 264005 P.R. China\\
				2. School of Mathematics and Statistics, Lanzhou University, Lanzhou 730000 P.R. China}
\end{center}

\vspace{1mm}

\begin{abstract}

Kemeny's constant for a connected graph $G$, denoted by $\mathcal{K}(G)$, is the expected time for a random walk to reach a randomly chosen vertex $u$, regardless of the choice of the initial vertex. Recently, Kim et al. (2026) proposed two conjectures on Kemeny's constant. The first conjecture asserts that if $G$ is a connected graph of order $n$ and diameter 2, then $\mathcal{K}(G) = O(n)$. The second conjecture asserts that if $G$ be a graph of order $n$, then $\min\{\mathcal{K}(G), \mathcal{K}(\overline{G})\} = O(n)$, and if both $G$ and $\overline{G}$ are connected, then $\mathcal{K}(G)\mathcal{K}(\overline{G}) = O(n^4)$, where 
$\overline{G}$ denotes the complement of $G$. In this paper, we confirm both conjectures. For the first conjecture, we prove that if $G$ is a connected graph of order $n$ and diameter 2, then
\[
\mathcal{K}(G) \leq (3 + \sqrt{5})(n - 1).
\]
For the second conjecture, we prove that for any \( n \)-vertex graph \( G \),
\[
\min\{\mathcal{K}(G), \mathcal{K}(\overline{G})\} \leq (8 + 2\sqrt{5})n - (10 + 2\sqrt{5}).
\]
Moreover, if both \( G \) and \( \overline{G} \) are connected, then
\[
\mathcal{K}(G)\mathcal{K}(\overline{G}) \leq \frac{3 + \sqrt{5}}{2}n^4.
\]
Our proof relies on effective estimates on resistance distances and spectral gaps of graphs.

\vspace{1mm}
\end{abstract}

\section{Introduction}
Let $G$ be a connected graph. The \textit{Kemeny's constant} of $G$, denoted by $\Kemeny(G)$, is a fundamental invariant arising from the theory of finite ergodic Markov chains and was first introduced by Kemeny and Snelll~\cite{Kemeny1969}. For a simple random walk on $G$, this constant measures the expected hitting time from a fixed vertex to a target vertex chosen independently according to the stationary distribution. Mathematically, it can be expressed as
\begin{equation}
\Kemeny(G)=\sum\limits_{j}\frac{d_G(j)}{2m}E_iT_j,
\end{equation}
where $d_G(j)$ is degree of vertex $j$ in $G$ and $E_iT_j$ is the expected hitting time from vertex $i$ to vertex $j$. A remarkable feature of Kemeny's constant is that  quantity is independent of the choice of the starting vertex $i$, making it a global measure of graph connectivity. It also has mathematical applications in graph theory,  in robotics \cite{Patel2015}, network science \cite{Altafini2023, Crisostomi2011,  Kells2020, Li2021}, and mathematical chemistry \cite{Li2019}.

Kemeny's constant is closely related to several important graph parameters, including effective graph resistance~\cite{Wang2017}, the Kirchhoff index~\cite{Palacios2021}, Randi\'{c} energy, and Laplacian incidence energy~\cite{Milovanovic2019}. Recall that the conventional distance function on a graph $G$ is the shortest-path distance, where the \textit{distance} $d_G(u,v)$ between two vertices $u,v\in V(G)$ is the length of a shortest path joining them. Another distance function, called the resistance distance, was introduced by Klein and Randi\'{c}~\cite{Klein1993}. For $u,v\in V(G)$, the \textit{resistance distance} $r_G(u,v)$ is defined as the effective resistance between the corresponding nodes in the electrical network obtained from $G$ by replacing each edge with a $1$-ohm resistor. It is known that $r_G(u,v)\leq d_G(u,v)$, with equality if and only if there is a unique path connecting $u$ and $v$ in $G$. In this context, Palacios and Renom~\cite{Palacios2011} established the following elegant identity relating Kemeny's constant to the resistance distances of a graph.
\begin{equation} \label{eq1.2}
			\Kemeny(G) = \frac{1}{4m} \sum_{u \in V(G)} \sum_{v \in V(G)} d_G(u)d_G(v) r_G(u,v).
\end{equation}

 As a fundamental parameter characterizing the structural properties of graphs, Kemeny's constant has recently attracted significant research attention. On the one hand, extremal values of Kemeny's constant have been investigated for various graph families in \cite{Breen2019, Jang2023, Liao2026, Yang2026}. On the other hand, the upper and lower bounds for Kemeny's constant have been extensively studied.
Brightwell and Winkler~\cite{Brightwell1990} established a tight general bound of $\Kemeny(G) = O(n^2)$ for any connected graph on $n$ vertices.
 Later,  Ciardo, Dahl and Kirkland~\cite{ciardo2020} established sharp upper and lower bounds for the Kemeny's constant of trees with fixed order and diameter, and characterized the extremal trees attaining these bounds. After that, Wang and Tao~\cite{wang2023}  derived the tight bounds for the Kemeny's constant of trees with a prescribed matching number, and identified the corresponding extremal structures. Recently, Kim et al.~\cite{Kim2026} studied Nordhaus--Gaddum problems for Kemeny's constant and obtained several bounds under degree conditions and for certain graph classes.

In graph theory, the diameter is one of the most fundamental parameters of a graph. The \emph{diameter} of a connected graph $G$, denoted by $\operatorname{diam}(G)$, is the maximum distance between any two vertices of $G$. Graphs of diameter two form a natural and important class. Despite the strong restriction on their diameter, these graphs exhibit considerable structural diversity, which has motivated extensive research on various graph invariants in this setting. For instance, Ma et al.~\cite{Ma2021} studied the pancyclicity and $1$-Hamiltonicity of line graphs of diameter-two graphs. Dubickas~\cite{Dubickas2021} subsequently established upper and lower bounds for the maximum total domination number of graphs of diameter two. Moreover, extremal graphs of diameter two have been determined for several resistance-distance-based invariants, including the Kirchhoff index~\cite{palacios2015}, the sum eccentric resistance distance~\cite{He2017}, and the multiplicative eccentricity resistance distance and its spectral radius~\cite{Hong2019}. More recently, Kim et al.~\cite{Kim2026} proposed the following conjecture concerning Kemeny's constant of graphs of diameter two.

\begin{conjecture}[\cite{Kim2026}]\label{conjecture1}
Let $G$ be a connected graph of order $n$ with $\diam(G)=2$. Then $\Kemeny(G)=O(n)$.
\end{conjecture}

A central direction in extremal graph theory is the Nordhaus--Gaddum problem, which studies the behavior of a graph parameter on a graph and its complement. Originally introduced for the chromatic number~\cite{Nordhaus1956}, Nordhaus--Gaddum questions have since been studied for a wide variety of graph invariants, including the matching number~\cite{Lin2017}, the second $A_\alpha$-eigenvalue~\cite{Chen2020}, the spectral radius~\cite{Hong2000}, the Kirchhoff index~\cite{Yang2011}, and the Wiener and Zagreb indices~\cite{Zhang2005}. For Kemeny's constant, Kim et al.~\cite{Kim2026} proposed the following conjecture. Following their convention, we set $\Kemeny(H)=\infty$ whenever $H$ is disconnected.

\begin{conjecture}[\cite{Kim2026}]\label{conjecture2}
Let $G$ be a graph of order $n$. Then
$$
\min\{\Kemeny(G),\Kemeny(\overline{G})\}=O(n).
$$
Moreover, if both $G$ and $\overline{G}$ are connected, then
$$
\Kemeny(G)\Kemeny(\overline{G})=O(n^4).
$$
\end{conjecture}

The results of this paper resolve both conjectures. First, for every connected graph $G$ of diameter two and every pair of distinct vertices $u,v\in V(G)$, we prove
$$
r_G(u,v)\leq(3+\sqrt5)\left(\frac{1}{d_G(u)}+\frac{1}{d_G(v)}\right).
$$
As a consequence, $\Kemeny(G)\leq(3+\sqrt5)(n-1)$, proving Conjecture~\ref{conjecture1}. We then resolve Conjecture~\ref{conjecture2} in full. In particular, for every graph $G$ of order $n\geq2$,
$$
\min\{\Kemeny(G),\Kemeny(\overline{G})\}\leq(8+2\sqrt5)n-(10+2\sqrt5),
$$
while for every connected complementary pair,
$$
\Kemeny(G)\Kemeny(\overline{G})\leq\frac{3+\sqrt5}{2}\,n^4.
$$
The proof of the minimum bound combines a structural property of complementary diameter-three graphs with the Nordhaus--Gaddum spectral-gap theorem of Kim and Madras~\cite{KimMadras2025}. More precisely, we show that a graph having a dominating set of size two satisfies a linear bound for all but the reciprocal of the second normalized Laplacian eigenvalue in the spectral expression for Kemeny's constant. This reduces the remaining diameter-three case to the spectral-gap estimate of Kim and Madras.

The remainder of the paper is organized as follows. In Section~2, we introduce the notation and recall standard facts about Dirichlet energy, voltage functions, the Laplacian operator, and the normalized Laplacian. In Section~3, we establish a key energy estimate for two disjoint vertex sets whose pairwise distances are at most two. In Section~4, we derive the effective-resistance bound for graphs of diameter two and its consequences for Kemeny's constant and the Kirchhoff index. In Section~5, we prove the Nordhaus--Gaddum minimum and product bounds and thereby resolve Conjecture~\ref{conjecture2}. Finally, in Section~6, we conclude the paper.

\section{Preliminaries}

In this section, we introduce the notation and facts used throughout the paper. All graphs considered are finite, simple, undirected, and unweighted. Unless stated otherwise, graphs are connected. Let $G=(V(G),E(G))$ be a graph of order $n$ and size $m$. The complement of $G$, denoted by $\overline{G}$, is the graph on $V(G)$ in which two distinct vertices are adjacent if and only if they are nonadjacent in $G$. For $v\in V(G)$, let $N_G(v)$ and $d_G(v)$ denote its neighborhood and degree, respectively. We write $\delta(G)$ and $\Delta(G)$ for the minimum and maximum degrees of $G$, respectively.

Let $A_G$ and $D_G$ denote the adjacency matrix and the diagonal degree matrix of $G$, respectively. The Laplacian matrix of $G$ is $L_G=D_G-A_G$. Since $G$ is connected, every vertex has positive degree, and the normalized Laplacian of $G$ is defined by
$$
\mathcal{L}_G=D_G^{-1/2}L_GD_G^{-1/2}=I-D_G^{-1/2}A_GD_G^{-1/2}.
$$
Write the eigenvalues of $\mathcal{L}_G$ as
$$
0=\lambda_1(\mathcal{L}_G)<\lambda_2(\mathcal{L}_G)\leq\cdots\leq\lambda_n(\mathcal{L}_G)\leq2.
$$
The transition matrix of the simple random walk on $G$ is $P_G=D_G^{-1}A_G$. Let its eigenvalues be
$$
1=\theta_1>\theta_2\geq\cdots\geq\theta_n\geq-1.
$$
Although $P_G$ is not necessarily symmetric, it is similar to the symmetric matrix $D_G^{-1/2}A_GD_G^{-1/2}$, since
$$
D_G^{1/2}P_GD_G^{-1/2}=D_G^{-1/2}A_GD_G^{-1/2}=I-\mathcal{L}_G.
$$
Therefore, $P_G$ and $I-\mathcal{L}_G$ have the same eigenvalues, and hence
$$
\lambda_i(\mathcal{L}_G)=1-\theta_i,\qquad 1\leq i\leq n.
$$
On the other hand, the standard spectral representation of Kemeny's constant gives
\begin{equation}\label{eq:kemeny-spectrum}
\Kemeny(G)=\sum_{i=2}^{n}\frac{1}{1-\theta_i}
=\sum_{i=2}^{n}\frac{1}{\lambda_i(\mathcal{L}_G)}.
\end{equation}
Moreover, the spectral gap of the simple random walk on $G$ is defined by $\operatorname{Gap}(G)=1-\theta_2$ \cite{KimMadras2025}. Consequently,
$$
\operatorname{Gap}(G)=\lambda_2(\mathcal{L}_G).
$$

\begin{definition}\label{def2.1}
Let $G$ be a connected graph and let $w:V(G)\to\mathbb{R}$ be a real-valued function. The Dirichlet energy of $w$ on $G$ is defined by
$$
\mathcal{E}_G(w)=\sum_{xy\in E(G)}\bigl(w(x)-w(y)\bigr)^2=w^{\top}L_Gw,
$$
where $w$ is viewed as a column vector indexed by $V(G)$.
\end{definition}

\begin{definition}
Let $G$ be a connected graph. A function $h:V(G)\to\mathbb{R}$ is said to be \emph{harmonic at} $x\in V(G)$ if
$$
h(x)=\frac{1}{d_G(x)}\sum_{y\in N_G(x)}h(y).
$$
\end{definition}

\begin{definition}
Let $G$ be a connected graph. For a function $\phi:V(G)\to\mathbb{R}$, the Laplacian operator is defined by
$$
(L_G\phi)(x)=\sum_{y\in N_G(x)}\bigl(\phi(x)-\phi(y)\bigr)
$$
for every $x\in V(G)$.
\end{definition}

\begin{definition}\label{definition:2.4}
Let $G$ be a connected graph and let $u,v\in V(G)$ be distinct. A voltage function associated with a unit current entering at $u$ and leaving at $v$ is a solution of $L_G\phi=e_u-e_v$, where $e_u$ is the unit vector at $u$. Such a solution is unique up to an additive constant.
\end{definition}

After the normalization $\phi(v)=0$, the value $\phi(u)$ equals the effective resistance $r_G(u,v)$, and
$$
\mathcal{E}_G(\phi)=\phi(u)-\phi(v)=r_G(u,v).
$$

Finally, we record the following maximum principle for harmonic functions on graphs.

\begin{proposition}\label{pro2.5}
Let $G$ be a connected graph, and let $\phi:V(G)\to\mathbb R$ be harmonic on $V(G)\setminus\{u,v\}$, where $u,v\in V(G)$ are distinct. Then
$$
\min\{\phi(u),\phi(v)\}\leq\phi(x)\leq\max\{\phi(u),\phi(v)\}\qquad\text{for every }x\in V(G).
$$
\end{proposition}

\section{Bounds on Kemeny’s constant for graphs of diameter two}
In this section, we first establish an energy estimate that will be used to bound the effective resistance in graphs of diameter two. The following lemma gives a lower bound for the Dirichlet energy of a function whose values are separated on two vertex sets with pairwise distance at most two.

\begin{lemma}\label{lem:capacity}
Let $G$ be a connected graph, where $A, B \subseteq V(G)$ are
nonempty disjoint vertex subsets with $|A| = p$ and $|B| = q$. Assume that $d_G(a,b)\le 2$ for every  $a\in A$ and $b\in B$.
Let $w:V(G)\to\mathbb R$ be an arbitrary real-valued function on the vertices such that $ w(a)\ge \alpha$ and $ w(b)\le \beta$ for some real numbers $\alpha>\beta$.  Then the Dirichlet energy of $w$ satisfies $\mathcal{E}_G({w})\ge \frac{pq}{p+q}(\alpha-\beta)^2$.

\end{lemma}

\begin{proof}
Set $g=(w-\beta)/(\alpha-\beta)$. Since $w\geq\alpha$ on $A$ and $w\leq\beta$ on $B$, we have $g\geq1$ on $A$ and $g\leq0$ on $B$. Moreover, for every edge $xy\in E(G)$,
$$
g(x)-g(y)=\frac{w(x)-\beta}{\alpha-\beta}-\frac{w(y)-\beta}{\alpha-\beta}=\frac{w(x)-w(y)}{\alpha-\beta}.
$$
It follows that $\mathcal{E}_G(g)=\mathcal{E}_G(w)/(\alpha-\beta)^2$. Therefore, after replacing $w$ by $g$, we may assume that $\alpha=1$ and $\beta=0$.
For each ordered pair $(a,b)\in A\times B$, we define a path $Q(a,b)$ of length at most two as follows. 

\noindent\textbf{Case 1.} Suppose that $a$ and $b$ are adjacent. Then we define $Q(a,b)$ to be the edge $ab$.

\noindent\textbf{Case 2.} Suppose that $a$ and $b$ are not adjacent. Since $\operatorname{diam}(G)\leq 2$, the vertices $a$ and $b$ have a common neighbor. Choose one such common neighbor $z$. We define $Q(a,b)$ according to the following three subcases:
\begin{itemize}
\item if $z\notin A\cup B$, then $Q(a,b)$ is the two-edge path $a$--$z$--$b$;
\item if $z\in A$, then $Q(a,b)$ is the edge $zb$;
\item if $z\in B$, then $Q(a,b)$ is the edge $az$.
\end{itemize}

In all cases, $Q(a,b)$ is either an edge joining a vertex of $A$ to a vertex of $B$, or a path of length two whose internal vertex lies outside $A \cup B$.
For a given $Q(a,b)$, we derive a lower bound on its energy contribution.

First suppose that $Q(a,b)$ is an $A$--$B$ edge $xy$, where $x\in A$ and
$y\in B$.  Since $g(x)\ge 1$ and $g(y)\le 0$, we have
\begin{equation}\label{eq:edge-choice}
        1\le \bigl(g(x)-g(y)\bigr)^2.
\end{equation}

Next suppose that $Q(a,b)$ is a two-edge path $a$-$z$-$b$ with
$z\notin A\cup B$.  Then
\[
        1\le g(a)-g(b)
        \le |g(a)-g(z)|+|g(z)-g(b)|.
\]

Applying the weighted Cauchy-Schwarz inequality
\[
        (X+Y)^2\le (p+q)\left(\frac{X^2}{q}+\frac{Y^2}{p}\right),
\]
with $ X = |g(a) - g(z)| $ and $ Y = |g(z) - g(b)| $, we get
\begin{equation}\label{eq:path-choice}
        1\le (p+q)
        \left(
        \frac{\bigl(g(a)-g(z)\bigr)^2}{q}
        +
        \frac{\bigl(g(z)-g(b)\bigr)^2}{p}
        \right).
\end{equation}

We now sum \eqref{eq:edge-choice} and \eqref{eq:path-choice} over all $pq$ ordered pairs $(a,b)\in A\times B$, and estimate the contribution of each edge according to the following three cases.

\noindent\textbf{Case 1.} Edges joining $A$ and $B$.

Let $ab$ be an edge with $a\in A$ and $b\in B$. It can be selected as $Q(a,b)$ in the direct case for at most one ordered pair. It can also be
selected in the case $z\in A$ only for ordered pairs whose second coordinate is its endpoint in $B$, hence for at most $p-1$ additional ordered pairs.  Similarly,
it can be selected in the case $z\in B$ for at most $q-1$ additional ordered pairs.  Hence any fixed $A$--$B$ edge is selected at most $p+q-1\le p+q$ times.
Therefore the  total contribution of all such edges is at most
\[
        (p+q)\sum_{\substack{xy\in E(G)\\ x\in A,\ y\in B}}
        \bigl(g(x)-g(y)\bigr)^2.
\]

\noindent\textbf{Case 2.} Edges from $A$ to outside $A\cup B$.

Let $az$ be an edge with $a\in A$ and $z\notin A\cup B$.  Such an edge can
occur as the first edge of $Q(a,b)$ only for ordered pairs with first coordinate
$a$.  There are at most $q$ such pairs.  Since each occurrence is weighted by
$(p+q)/q$ in \eqref{eq:path-choice}, the total contribution of this fixed edge
is at most
\[
        q\cdot \frac{p+q}{q}\bigl(g(a)-g(z)\bigr)^2
        =(p+q)\bigl(g(a)-g(z)\bigr)^2.
\]

\noindent\textbf{Case 3.} Edges from $A\cup B$ to outside $B$.

Let $zb$ be an edge with $z\notin A\cup B$ and $b\in B$. In the two-edge path $Q(a,b)$ = $a$-$z$-$b$, this edge must serve as the second edge. It can be selected only for ordered pairs whose second coordinate is $b$, hence at most $p$ times; since each occurrence is weighted by
$(p+q)/p$, its total contribution is at most
\[
        (p+q)\bigl(g(z)-g(b)\bigr)^2.
\]

Combining these estimates gives
\[
pq \leq (p+q)\sum_{xy\in E(G)}\big(g(x) - g(y)\big)^2 = (p+q)\,\mathcal{E}_{G}(g).
\]
Thus
\[
\mathcal{E}_{G}(g)\geq \frac{pq}{p + q}.
\]
Multiplying by \((\alpha - \beta)^{2}\) and returning to the original function \(w\) yields the claimed inequality
\[
\mathcal{E}_{G}(w)\geq \frac{pq}{p + q}(\alpha - \beta)^{2}.
\]
This completes the proof.
\end{proof}

\section{Bounds on Kemeny's constant for graphs of diameter two}
 
In the following, we  establish an upper bound for the effective resistance between any two vertices of a connected graph of diameter two in terms of their degrees. We then apply this resistance estimate to derive a linear upper bound for Kemeny's constant and related  bound for the Kirchhoff index of connected graphs of diameter two.
\begin{theorem}\label{thm:diam2-resistance}
Let $G$ be a connected graph with $\diam(G)=2$.  Then for every two
distinct vertices $u,v\in V(G)$,
\[
        r_G(u,v)\le (3+\sqrt5)\left(\frac1{d_G(u)}+\frac1{d_G(v)}\right).
\]
\end{theorem}

\begin{proof}
Fix two distinct vertices $u,v\in V(G)$, and set $R:=r_G(u, v)$. Let $\phi:V(G)\to\mathbb R$ be the voltage function induced by injecting one unit of current at
$u$ and extracting it at $v$, normalized by $\phi(v)=0$.  Then $\phi$ satisfies
\[
L_G\phi = e_u - e_v, \qquad \phi(u)=R.
\]
By Definition~\ref{def2.1}, the Dirichlet energy is the quadratic form $\mathcal{E}_G(\phi) = \phi^\top L_G \phi$. Left-multiplying the voltage equation by $\phi^\top$ yields

\begin{equation}\label{eq:energy-is-R}
\mathcal{E}_G(\phi)=\phi^\top L_G\phi = \phi^\top(e_u-e_v)=\phi(u)-\phi(v)=R.
\end{equation}
Furthermore, by Proposition~\ref{pro2.5}, since $\phi$ is harmonic on $V(G)\backslash\{u,v\}$ with boundary values $R$ at $u$ and $0$ at $v$, we have

\begin{equation}\label{eq:max-principle}
        0\le \phi(x)\le R
        \quad\text{for every }x\in V(G).
\end{equation}

At vertex $u$, we have
\[
        1=(L_G\phi)(u)=\sum_{x\in N_G(u)}\bigl(\phi(u)-\phi(x)\bigr)
        =\sum_{x\in N_G(u)}\bigl(R-\phi(x)\bigr).
\]
Each summand is nonnegative by \eqref{eq:max-principle}.  Therefore at least
$d_G(u)/2$ vertices $x\in N_G(u)$ satisfy
\[
        R-\phi(x)\le \frac2{d_G(u)}.
\]
Define
\[
        S=\left\{x\in N_G(u):\phi(x)\ge R-\frac2{d_G(u)}\right\}.
\]
Then
\begin{equation}\label{eq:S-size}
        |S|\ge \frac{d_G(u)}{2}.
\end{equation}

Similarly, at vertex $v$,
\[
        -1=(L_G\phi)(v)=\sum_{y\in N_G(v)}\bigl(\phi(v)-\phi(y)\bigr)
        =-\sum_{y\in N_G(v)}\phi(y),
\]
so
\[
        1=\sum_{y\in N_G(v)}\phi(y).
\]
Again all summands are nonnegative by \eqref{eq:max-principle}.  Hence at least
${d_G(v)}/2$ vertices $y\in N_G(v)$ satisfy
\[
        \phi(y)\le \frac2{d_G(v)}.
\]
Define
\[
        T=\left\{y\in N_G(v):\phi(y)\le \frac2{d_G(v)}\right\}.
\]
Then
\begin{equation}\label{eq:T-size}
        |T|\ge \frac{d_G(v)}{2}.
\end{equation}

Let $s:=1/d_G(u)+1/d_G(v)$. If $R\le 2s$, then the desired inequality follows immediately because $2<3+\sqrt5$. Hence we may assume
 \begin{equation}\label{eq:R-larger-than-2s}
        R>2s.
\end{equation}
 Under this assumption, we know that $S\cap T=\varnothing$; otherwise, if $z\in S\cap T$, then
\[
        R-\frac2{d_G(u)}\le \phi(z)\le \frac2{d_G(v)},
\]
which implies $R\le 2/d_G(u)+2/d_G(v)=2s$, contradicting the assumption \eqref{eq:R-larger-than-2s}.

For every $x\in S$ and $y\in T$, we have
\[
        \phi(x)-\phi(y)
        \ge R-\frac2{d_G(u)}-\frac2{d_G(v)}
        = R-2s.
\]
Set $\gamma:=R-2s>0$. Since $\diam(G)=2$, every $x\in S$ and $y\in T$ satisfy $d_G(x,y)\le 2$.
Applying Lemma~\ref{lem:capacity} to $A=S$, $B=T$, $w=\phi$,
$\alpha=R-2/d_G(u)$, and $\beta=2/d_G(v)$, we obtain
\begin{equation}\label{eq:energy-lower-ST}
        \E_G(\phi)\ge \frac{|S||T|}{|S|+|T|}\gamma^2.
\end{equation}
The function $(p,q)\mapsto pq/(p+q)$ is increasing in each variable for
$p,q>0$.  Hence \eqref{eq:S-size} and \eqref{eq:T-size} imply
\begin{equation}\label{eq:harmonic-mean-lower}
        \frac{|S||T|}{|S|+|T|}
        \ge
        \frac{(d_G(u)/2)(d_G(v)/2)}{d_G(u)/2+d_G(v)/2}
        =\frac{d_G(u)d_G(v)}{2(d_G(u)+d_G(v))}
        =\frac1{2s}.
\end{equation}
Combining \eqref{eq:energy-is-R}, \eqref{eq:energy-lower-ST}, and
\eqref{eq:harmonic-mean-lower} gives
\begin{equation}\label{eq:R-2s}
        R\ge \frac{(R-2s)^2}{2s}.
\end{equation}
Let $X:=R/s$. Since $R>2s$, it follows that $X>2$, and inequality \eqref{eq:R-2s} becomes
\[
        X\ge \frac{(X-2)^2}{2},
\]
which is equivalent to $X^2-6X+4\le 0$. The roots of this quadratic  are $3-\sqrt5$ and $3+\sqrt5$;  together with the condition $X>2$, we conclude
that $X\le 3+\sqrt5$. Therefore
\[
        R\le (3+\sqrt5)s
        =(3+\sqrt5)\left(\frac1{d_G(u)}+\frac1{d_G(v)}\right),
\]
as claimed. This completes the proof.
\end{proof}

\begin{corollary}\label{cor:Kemeny-diam2}
Let $G$ be a connected  graph on $n$ vertices with $m$ edges and $\diam(G)=2$.
Then
\[
        \Kemeny(G)\le (3+\sqrt5)(n-1).
\]
In particular,
\[
        \Kemeny(G)=O(n).
\]
\end{corollary}

\begin{proof}
Using Eq. \eqref{eq1.2} and Theorem~\ref{thm:diam2-resistance},
\begin{align*}
\Kemeny(G)
&=\frac{1}{4m}\sum_{x \in V(G)} \sum_{y \in V(G)} d_G(x)d_G(y) r_G(x,y) \\
&=\frac{1}{4m}\sum_{\substack{x,y\in V(G)\\x\ne y}}d_G(x)d_G(y) r_G(x,y) \\
&\le \frac{3+\sqrt5}{4m}
\sum_{\substack{x,y\in V(G)\\x\ne y}}
        d_G(x)d_G(y)\left(\frac1{d_G(x)}+\frac1{d_G(y)}\right) \\
&=\frac{3+\sqrt5}{4m}
\sum_{\substack{x,y\in V(G)\\x\ne y}}(d_G(x)+d_G(y)).
\end{align*}
Now
\[
        \sum_{\substack{x,y\in V(G)\\x\ne y}}d_G(x)
        =(n-1)\sum_{x\in V(G)}d_G(x)
        =2m(n-1),
\]
and the same identity holds with $d_G(y)$ in place of $d_G(x)$.  Hence
\[
        \sum_{\substack{x,y\in V(G)\\x\ne y}}(d_G(x)+d_G(y))
        =4m(n-1).
\]
Therefore
\[
        \Kemeny(G)\le (3+\sqrt5)(n-1).
\]
\end{proof}

We next consider two resistance-based invariants that are closely related to Kemeny's constant, namely the Kirchhoff index and the multiplicative degree-Kirchhoff index. The Kirchhoff index $Kf(G)$  of $G$ is defined as
\[
Kf(G)=\frac{1}{2} \sum_{u \in V(G)} \sum_{v \in V(G)} r_G(u,v),
\]
and the multiplicative degree-Kirchhoff index  $Kf^{*}(G)$  of $G$ is defined as
\[
Kf^{*}(G)=\frac{1}{2} \sum_{u \in V(G)} \sum_{v \in V(G)}d_G(u)d_G(v)r_G(u,v).
\]

\begin{corollary}\label{cor:degree-kirchhoff}
Let $G$ be a connected  graph on $n$ vertices with $m$ edges and $\diam(G)=2$. Then
\[
        Kf^{*}(G)\le 2(3+\sqrt5)m(n-1).
\]
\end{corollary}

\begin{proof}
This is just the identity $Kf^{*}(G)=2m\Kemeny(G)$ together with
Corollary~\ref{cor:Kemeny-diam2}.
\end{proof}

Palacios~\cite{palacios2015} proved that every connected graph $G$ of order $n$ and diameter two satisfies $Kf(G)\leq(n-1)^2$, with equality if and only if $G\cong S_n$, where $S_n$ denotes the star on $n$ vertices. As a further consequence of our effective-resistance estimate, we obtain the following upper bound in terms of the vertex degrees.

\begin{corollary}\label{cor:Kirchhoff-diam2}
Let $G$ be a connected graph of order $n$ with $\diam(G)=2$. Then
\[
{Kf}(G)\leq (3+\sqrt{5})(n-1)\sum_{x\in V(G)}\frac{1}{d_G(x)}
\leq (3+\sqrt{5})\frac{n(n-1)}{\delta(G)}.
\]
In particular, $Kf(G)=O(n)$ for every family of such graphs satisfying $\delta(G)=\Omega(n)$.
\end{corollary}

\begin{proof}
By Theorem~\ref{thm:diam2-resistance}, we have
\begin{align*}
{Kf}(G)
&=\sum_{\{u,v\}\in\binom{V(G)}{2}}r_G(u,v)\\
&\leq (3+\sqrt{5})\sum_{\{u,v\}\in\binom{V(G)}{2}}
\left(\frac{1}{d_G(u)}+\frac{1}{d_G(v)}\right)\\
&=(3+\sqrt{5})(n-1)\sum_{u\in V(G)}\frac{1}{d_G(u)}\\
&\leq (3+\sqrt{5})\frac{n(n-1)}{\delta(G)}.
\end{align*}
The final assertion follows immediately when $\delta(G)=\Omega(n)$.
\end{proof}

\section{Nordhaus--Gaddum  bounds for Kemeny's constant}

In this section, we resolve Conjecture~\ref{conjecture2}. We first record two elementary diameter estimates. We then establish the minimum bound by combining a new estimate for graphs with a dominating pair with the spectral-gap theorem of Kim and Madras~\cite{KimMadras2025}, and finally prove the product bound.

\begin{lemma}\label{lem:general-diameter-bound}
Let $G$ be a connected graph with $m$ edges and diameter $d$. Then $\Kemeny(G)\leq md$.
In particular, every connected graph $G$ of order $n$ satisfies $\Kemeny(G)\leq\binom{n}{2}(n-1)=\frac12n(n-1)^2$.
\end{lemma}

\begin{proof}
For any two vertices $x,y\in V(G)$, the effective resistance is at most the graph distance, and hence $r_G(x,y)\leq d_G(x,y)\leq d$. By \eqref{eq1.2},
\begin{align*}
\Kemeny(G)
&=\frac{1}{4m}\sum_{x\in V(G)}\sum_{y\in V(G)}d_G(x)d_G(y)r_G(x,y)\\
&\leq\frac{d}{4m}\sum_{x\in V(G)}\sum_{y\in V(G)}d_G(x)d_G(y)\\
&=\frac{d}{4m}\left(\sum_{x\in V(G)}d_G(x)\right)^2
=\frac{d}{4m}(2m)^2
=md.
\end{align*}
The second assertion follows from $d\leq n-1$ and $m\leq\binom{n}{2}$.
\end{proof}

\begin{lemma}[\cite{Bondy2008}]\label{lem:diam-complement}
Let $G$ be a connected graph. If $\diam(G)>3$, then $\diam(\overline{G})\leq2$.
\end{lemma}

\begin{proof}
Take any two distinct vertices $x,y\in V(G)$. If $xy\notin E(G)$, then $x$ and $y$ are adjacent in $\overline{G}$. Suppose that $xy\in E(G)$. We claim that some vertex $z$ is nonadjacent to both $x$ and $y$ in $G$. Otherwise, every vertex of $G$ would be adjacent to at least one of $x$ and $y$, and then any two vertices of $G$ would be joined by a path of length at most $3$, contradicting $\diam(G)>3$. Thus such a vertex $z$ exists, and $xzy$ is a path of length two in $\overline{G}$. Hence $\diam(\overline{G})\leq2$.
\end{proof}

\subsection{The minimum bound}

The proof of the first assertion of Conjecture~\ref{conjecture2} uses the following theorem of Kim and Madras~\cite{KimMadras2025}. In the notation introduced in Section~2, their spectral gap is $\lambda_2(\mathcal{L}_G)$.

\begin{lemma}[\cite{KimMadras2025}]\label{lem:KM-gap}
Let $G$ be a graph of order $n\geq2$ such that both $G$ and $\overline{G}$ are connected. Then
$$
\max\{\lambda_2(\mathcal{L}_G),\lambda_2(\mathcal{L}_{\overline{G}})\}\geq\frac{3-\sqrt5}{8(n-1)}.
$$
\end{lemma}

Recall that a set $S\subseteq V(H)$ is a \emph{dominating set} if every vertex in $V(H)\setminus S$ has a neighbor in $S$. We first show that a dominating set of size two controls all but the first nonzero term in the spectral formula \eqref{eq:kemeny-spectrum}.

\begin{lemma}\label{lem:dominating-pair-kemeny}
Let $H$ be a connected graph of order $n\geq3$ having a dominating set of size two. Then
$$
\Kemeny(H)\leq\frac{1}{\lambda_2(\mathcal{L}_H)}+2(n-2).
$$
\end{lemma}

\begin{proof}
Let $S=\{a,b\}$ be a dominating set of $H$, put $W=V(H)\setminus S$, and let $M=\mathcal{L}_H[W]$ be the principal submatrix of the normalized Laplacian indexed by $W$. Write $0=\lambda_1(\mathcal{L}_H)<\lambda_2(\mathcal{L}_H)\leq\cdots\leq\lambda_n(\mathcal{L}_H)$ for the eigenvalues of $\mathcal{L}_H$, and let $0<\mu_1\leq\cdots\leq\mu_{n-2}$ be the eigenvalues of $M$.

Let $Q=L_H[W]$. The matrix $Q$ is positive definite. Indeed, if $z^{\top}Qz=0$, then extending $z$ by zero on $S$ produces a function on $V(H)$ with zero Dirichlet energy. Since $H$ is connected and the extension vanishes on $S$, it must be identically zero. Thus $Q$, and hence $M$, is positive definite.

By the Cauchy interlacing theorem,
$$
\lambda_i(\mathcal{L}_H)\leq\mu_i\leq\lambda_{i+2}(\mathcal{L}_H)\qquad(1\leq i\leq n-2).
$$
Therefore
$$
\sum_{i=3}^{n}\frac{1}{\lambda_i(\mathcal{L}_H)}\leq\sum_{i=1}^{n-2}\frac{1}{\mu_i}=\operatorname{tr}(M^{-1}).
$$
Using \eqref{eq:kemeny-spectrum}, we obtain
\begin{equation}\label{eq:dom-pair-first}
\Kemeny(H)\leq\frac{1}{\lambda_2(\mathcal{L}_H)}+\operatorname{tr}(M^{-1}).
\end{equation}

Let $D_W=\operatorname{diag}(d_H(x):x\in W)$. Since $M=D_W^{-1/2}QD_W^{-1/2}$, we have $M^{-1}=D_W^{1/2}Q^{-1}D_W^{1/2}$ and hence
\begin{equation}\label{eq:trace-grounded}
\operatorname{tr}(M^{-1})=\sum_{x\in W}d_H(x)(Q^{-1})_{xx}.
\end{equation}
For a fixed $x\in W$, solving $Qh=e_x$ and extending $h$ by zero on $S$ gives the voltage produced by injecting one unit of current at $x$ and grounding the vertices of $S$. Equivalently, the vertices of $S$ are wired together and serve as a single grounded terminal. Consequently, $(Q^{-1})_{xx}=h(x)$ is the effective resistance from $x$ to the wired set $S$.

Fix $x\in W$, and set $p_x=|N_H(x)\cap W|$ and $q_x=|N_H(x)\cap S|$. Since $S$ is dominating, $q_x\geq1$, and $d_H(x)=p_x+q_x$. For every $y\in N_H(x)\cap W$, choose a vertex $s_y\in S$ adjacent to $y$, which is possible because $S$ dominates $H$. Consider the subnetwork consisting of the $q_x$ edges from $x$ directly to $S$ and, for each $y\in N_H(x)\cap W$, the two-edge path $x$--$y$--$s_y$. After the vertices of $S$ are identified, these are parallel branches: $q_x$ branches of resistance $1$ and $p_x$ branches of resistance $2$. The effective resistance of this subnetwork from $x$ to $S$ is therefore
$$
\frac{1}{q_x+p_x/2}.
$$
After wiring the vertices of $S$, this subnetwork is obtained by deleting all other edges. Since deleting edges can only increase effective resistance, Rayleigh monotonicity~\cite{Doyle1984} gives
$$
(Q^{-1})_{xx}\leq\frac{1}{q_x+p_x/2}\leq\frac{2}{p_x+q_x}=\frac{2}{d_H(x)}.
$$
Substituting this estimate into \eqref{eq:trace-grounded} yields $\operatorname{tr}(M^{-1})\leq2|W|=2(n-2)$. Together with \eqref{eq:dom-pair-first}, this proves the lemma.
\end{proof}

The next elementary observation explains why dominating pairs arise naturally in the remaining diameter-three case.

\begin{lemma}\label{lem:diam3-dominating-pair}
If $\diam(\overline{H})=3$, then $H$ has a dominating set of size two.
\end{lemma}

\begin{proof}
Choose vertices $a,b\in V(H)$ with $d_{\overline{H}}(a,b)=3$. Since $a$ and $b$ are nonadjacent in $\overline{H}$, they are adjacent in $H$. We claim that $\{a,b\}$ dominates $H$. Otherwise, some vertex $x\notin\{a,b\}$ is adjacent to neither $a$ nor $b$ in $H$. Then $x$ is adjacent to both $a$ and $b$ in $\overline{H}$, so $a$--$x$--$b$ is a path of length two in $\overline{H}$, contradicting $d_{\overline{H}}(a,b)=3$.
\end{proof}

\begin{theorem}\label{thm:NG-min}
Let $G$ be a graph of order $n\geq2$, with the convention that $\Kemeny(H)=\infty$ when $H$ is disconnected. Then
$$
\min\{\Kemeny(G),\Kemeny(\overline{G})\}\leq(8+2\sqrt5)n-(10+2\sqrt5).
$$
In particular, $\min\{\Kemeny(G),\Kemeny(\overline{G})\}=O(n)$.
\end{theorem}

\begin{proof}
We first consider the case in which one of $G$ and $\overline{G}$ is disconnected. Suppose, without loss of generality, that $G$ is disconnected. Then $\overline{G}$ is connected and has diameter at most two. Indeed, if two vertices lie in different components of $G$, then they are adjacent in $\overline{G}$; if they lie in the same component, choose a vertex in another component, which gives a path of length two between them in $\overline{G}$. If $\diam(\overline{G})=2$, Corollary~\ref{cor:Kemeny-diam2} gives $\Kemeny(\overline{G})\leq(3+\sqrt5)(n-1)$. If $\diam(\overline{G})=1$, then $\overline{G}=K_n$. Since the normalized Laplacian eigenvalues of $K_n$ are $0$ and $n/(n-1)$ with multiplicity $n-1$, \eqref{eq:kemeny-spectrum} gives $\Kemeny(K_n)=(n-1)^2/n<(3+\sqrt5)(n-1)$. Thus in either case $\min\{\Kemeny(G),\Kemeny(\overline{G})\}\leq(3+\sqrt5)(n-1)$. Since $(3+\sqrt5)(n-1)\leq(8+2\sqrt5)n-(10+2\sqrt5)$ for $n\geq2$, the claimed bound follows. The case in which $\overline{G}$ is disconnected is symmetric.

We may now assume that both $G$ and $\overline{G}$ are connected. If either graph has diameter two, Corollary~\ref{cor:Kemeny-diam2} and the preceding comparison again give the desired bound. Hence suppose that both diameters are at least three. Lemma~\ref{lem:diam-complement} then implies $\diam(G)=\diam(\overline{G})=3$. Applying Lemma~\ref{lem:diam3-dominating-pair} first to $H=G$ and then to $H=\overline{G}$ shows that both $G$ and $\overline{G}$ have dominating sets of size two.

By Lemma~\ref{lem:KM-gap}, one of $G$ and $\overline{G}$, say $H$, satisfies
$$
\lambda_2(\mathcal{L}_H)\geq\frac{3-\sqrt5}{8(n-1)}.
$$
Since $H$ has a dominating set of size two, Lemma~\ref{lem:dominating-pair-kemeny} yields
\begin{align*}
\Kemeny(H)
&\leq\frac{1}{\lambda_2(\mathcal{L}_H)}+2(n-2)\\
&\leq\frac{8(n-1)}{3-\sqrt5}+2(n-2)\\
&=2(3+\sqrt5)(n-1)+2(n-2)\\
&=(8+2\sqrt5)n-(10+2\sqrt5).
\end{align*}
Since $\min\{\Kemeny(G),\Kemeny(\overline{G})\}\leq\Kemeny(H)$, the proof is complete.
\end{proof}

\subsection{The product bound}

\begin{theorem}\label{thm:NG-product}
Let $G$ be a graph of order $n$ such that both $G$ and $\overline{G}$ are connected. Then
$$
\Kemeny(G)\Kemeny(\overline{G})\leq\frac{3+\sqrt5}{2}\,n^4.
$$
In particular, $\Kemeny(G)\Kemeny(\overline{G})=O(n^4)$.
\end{theorem}

\begin{proof}
We distinguish two cases according to the minimum of the diameters of $G$ and $\overline{G}$.

\noindent\textbf{Case 1.} Suppose that $\min\{\diam(G),\diam(\overline{G})\}=2$. Without loss of generality, assume that $\diam(G)=2$. By Corollary~\ref{cor:Kemeny-diam2}, $\Kemeny(G)\leq(3+\sqrt5)(n-1)$. By Lemma~\ref{lem:general-diameter-bound},
$$
\Kemeny(\overline{G})\leq\frac12n(n-1)^2.
$$
Therefore
$$
\Kemeny(G)\Kemeny(\overline{G})\leq\frac{3+\sqrt5}{2}\,n(n-1)^3<\frac{3+\sqrt5}{2}\,n^4.
$$

\noindent\textbf{Case 2.} Suppose that $\min\{\diam(G),\diam(\overline{G})\}\geq3$. If $\diam(G)>3$, then Lemma~\ref{lem:diam-complement} gives $\diam(\overline{G})\leq2$, a contradiction. By symmetry, $\diam(\overline{G})>3$ is also impossible. Hence $\diam(G)=\diam(\overline{G})=3$. Let $m=|E(G)|$ and $\overline{m}=|E(\overline{G})|$. Lemma~\ref{lem:general-diameter-bound} gives $\Kemeny(G)\leq3m$ and $\Kemeny(\overline{G})\leq3\overline{m}$. Since $m+\overline{m}=\binom{n}{2}$,the arithmetic-geometric mean inequality gives
$$
m\overline{m}\leq\left(\frac{m+\overline{m}}{2}\right)^2=\frac14\binom{n}{2}^2.
$$
Consequently,
$$
\Kemeny(G)\Kemeny(\overline{G})\leq9m\overline{m}\leq\frac94\binom{n}{2}^2=\frac{9}{16}n^2(n-1)^2<\frac{9}{16}n^4.
$$
Since $\frac{9}{16}<\frac{3+\sqrt5}{2}$, the desired bound follows.
\end{proof}

Theorem~\ref{thm:NG-min} proves the first assertion of Conjecture~\ref{conjecture2}, while Theorem~\ref{thm:NG-product} proves the second. Hence Conjecture~\ref{conjecture2} is completely resolved.

\section{Concluding remarks}

In this paper, we have resolved two conjectures proposed by Kim et al.~\cite{Kim2026} concerning Kemeny's constant of graphs. First, we proved that every connected graph $G$ of diameter two satisfies $\Kemeny(G)=O(n)$ by establishing a pointwise effective-resistance estimate in terms of the endpoint degrees. Second, we resolved their Nordhaus--Gaddum conjecture in full: for every graph $G$ of order $n$, $\min\{\Kemeny(G),\Kemeny(\overline{G})\}=O(n)$, and whenever both $G$ and $\overline{G}$ are connected, $\Kemeny(G)\Kemeny(\overline{G})=O(n^4)$.

Two different mechanisms underlie these results. The diameter-two theorem is driven by the energy lower bound of Section~3, which converts local voltage separation into a global effective-resistance estimate. For the Nordhaus--Gaddum minimum problem, the only nontrivial connected case remaining after the diameter reduction is $\diam(G)=\diam(\overline{G})=3$. In this situation, each graph has a dominating set of size two. Lemma~\ref{lem:dominating-pair-kemeny} shows that such a dominating pair controls all terms in the spectral representation of Kemeny's constant except the reciprocal spectral-gap term, and the Nordhaus--Gaddum theorem of Kim and Madras~\cite{KimMadras2025} controls that remaining term.

It remains natural to determine the best possible leading constant in a universal linear bound for $\min\{\Kemeny(G),\Kemeny(\overline{G})\}$ and to identify graph families that are extremal or asymptotically extremal for this problem. It would also be interesting to sharpen the constant in the effective-resistance bound for diameter-two graphs.

\section{Declaration of competing interest}
		The authors declare that they have no known competing financial interests or personal relationships that could have appeared to influence the work reported in this paper.
\section{Data availability}
		No data was used for the research described in the article.
\section{Acknowledgments}
	    The support of the National Natural Science Foundation of China (through grant no. 12171414), Taishan Scholars Special Project of Shandong Province and Graduate Innovation Foundation of Yantai University (through grant no. GGIFYTU2616), is greatly acknowledged.

\end{document}